\documentclass[11pt,a4paper]{article}

\usepackage{amsmath,amsthm,amssymb,color}
\usepackage[authoryear]{natbib}

\usepackage{bbm}
\usepackage{mathrsfs}
\usepackage[breaklinks=true]{hyperref}

\numberwithin{equation}{section}
\allowdisplaybreaks[4]

\theoremstyle{plain}
\newtheorem{theorem}{Theorem}[section]
\newtheorem{proposition}[theorem]{Proposition}
\newtheorem{lemma}[theorem]{Lemma}

\theoremstyle{definition}

\newtheorem{remark}[theorem]{Remark}

\makeatletter

\newcount\minute
\newcount\hour
\newcount\hourMins
\def\now{%
\minute=\time%
\hour=\time \divide \hour by 60%
\hourMins=\hour \multiply\hourMins by 60%
\advance\minute by -\hourMins%
\zeroPadTwo{\the\hour}:\zeroPadTwo{\the\minute}%
}
\def\zeroPadTwo#1{\ifnum #1<10 0\fi#1}

\renewcommand{\cite}{\citet*}

\def\^#1{\ifmmode {\mathaccent"705E #1} \else {\accent94 #1} \fi}
\def\~#1{\ifmmode {\mathaccent"707E #1} \else {\accent"7E #1} \fi}

\def\*#1{#1^\ast}
\edef\-#1{\noexpand\ifmmode {\noexpand\bar{#1}} \noexpand\else \-#1\noexpand\fi}
\def\>#1{\vec{#1}}
\def\.#1{\dot{#1}}

\def\atop{\@@atop}
\def\*#1{\mathscr{#1}}

\renewcommand{\leq}{\leqslant}
\renewcommand{\geq}{\geqslant}
\renewcommand{\phi}{\varphi}

\newcommand{\eq}{\eqref}

\newcommand{\IE}{\mathbbm{E}}

\newcommand{\Var}{\mathop{\mathrm{Var}}\nolimits}
\newcommand{\Cov}{\mathop{\mathrm{Cov}}}

\def\be#1{\begin{equation*}#1\end{equation*}}
\def\ben#1{\begin{equation}#1\end{equation}}

\def\besn#1{\begin{equation}\begin{split}#1\end{split}\end{equation}}

\def\beqn#1\eeqn{\begin{align}#1\end{align}}
\def\beq#1\eeq{\begin{align*}#1\end{align*}}

\usepackage{graphicx}
\usepackage{latexsym}
\usepackage{amsmath,amsthm,amssymb,amscd}
\usepackage{epsf,amsmath}

\def\E{{\IE}}

\def\I{{\rm I}}

\renewcommand\section{\@startsection {section}{1}{\z@}%
{-3.5ex \@plus -1ex \@minus -.2ex}%
{1.3ex \@plus.2ex}%
{\center\small\sc\mathversion{bold}\MakeUppercase}}

\def\subsection#1{\@startsection {subsection}{2}{0pt}%
{-3.5ex \@plus -1ex \@minus -.2ex}%
{1ex \@plus.2ex}%
{\bf\mathversion{bold}}{#1}}

\def\subsubsection#1{\@startsection{subsubsection}{3}{0pt}%
{\medskipamount}%
{-10pt}%
{\normalsize\itshape}{\kern-2.2ex. #1.}}

\def\blfootnote{\xdef\@thefnmark{}\@footnotetext}

\makeatother

\begin{document}

\title{\sc\bf\large\MakeUppercase{A universal error bound in the CLT for counting monochromatic edges in uniformly colored graphs}}
\author{\sc Xiao Fang \footnote{Department of Statistics and Applied Probability, 
Block S16, Level 7, 6 Science Drive 2,
Faculty of Science, 
National University of Singapore,
Singapore 117546}}
\date{\it National University of Singapore and Stanford University}
\maketitle

\begin{abstract} 
Let $\{G_n: n\geq 1\}$ be a sequence of simple graphs. Suppose $G_n$ has $m_n$ edges and each vertex of $G_n$ is colored independently and uniformly at random with $c_n$ colors. Recently, Bhattacharya, Diaconis and Mukherjee (2013) proved universal limit theorems for the number of monochromatic edges in $G_n$. Their proof was by the method of moments, and therefore was not able to produce rates of convergence. By a non-trivial application of Stein's method, we prove that there exists a universal error bound for their central limit theorem. The error bound depends only on $m_n$ and $c_n$, regardless of the graph structure.

\end{abstract}

\vspace{9pt}
\par

\vspace{9pt}
\par

\section{Introduction}

Let $\{G_n: n\geq 1\}$ be a sequence of simple graphs, that is, graphs that contain no loops and no multiple edges. Suppose $G_n$ has $m_n$ edges and each vertex of $G_n$ is colored independently and uniformly at random with $c_n$ colors. Let $Y_n$ be the number of monochromatic edges in $G_n$. Using the coupling approach in Stein's method for Poisson approximation, \cite{BaHoJa92} (page 105, Theorem 5.G) proved that 
\ben{\label{13}
d_{TV}(\mathcal{L}(Y_n), Poi(\frac{m_n}{c_n}))\leq \frac{\sqrt{8 m_n}}{c_n}
}
where $d_{TV}$ denotes the total variation distance and $Poi(\lambda)$ denotes the Poisson distribution with mean $\lambda$.
The bound \eq{13} implies that 
if $c_n\rightarrow \infty$ and $m_n/c_n \rightarrow \lambda>0$, the distribution of $Y_n$ converges to the Poisson distribution with mean $\lambda$. Recently, \cite{BhDiMu13} reproved this Poisson limit theorem by the method of moments. By the same method, they also showed that in the case $c_n\rightarrow \infty$ and $m_n/c_n \rightarrow \infty$, the distribution of $W_n$, after proper standardization, converges to the standard normal distribution.
These limit theorems were called \emph{universal} limit theorems because they do not require any assumption on the graph structure.
For applications of this and related problems, we refer to \cite{BhDiMu13} and the references therein.

In this note, we prove the following result.
\begin{theorem}\label{t2}
Let $Y$ be the number of monochromatic edges in a simple graph with $m$ edges where each vertex is colored independently and uniformly at random with $c$ colors.
Let
\be{
W=\frac{(Y-\frac{m}{c})}{\sqrt{\frac{m}{c}(1-\frac{1}{c})}}.
}
We have
\ben{\label{t2-1}
d_{W}(\mathcal{L}(W), N(0,1))\leq 
\frac{3}{2}\sqrt{\frac{c}{m}}+\frac{5\sqrt{2}}{\sqrt{c}}+\frac{1}{\sqrt{\pi}}\frac{2^{7/4}}{m^{1/4}}.
}
where $d_W$ denotes the Wasserstein distance and $N(0,1)$ denotes the standard normal distribution.
\end{theorem}
The bound \eq{t2-1} provides a universal error bound for the central limit theorem for $W_n$ as $c_n\rightarrow \infty$ and $m_n/c_n \rightarrow \infty$. A corollary for fixed $c_n$ is also obtained in Remark \ref{r1}. The bound \eq{t2-1} is obtained by a non-trivial application of Stein's method for normal approximation.
Stein's method was introduced by \cite{St72} for normal approximation. Stein's method for Poisson approximation was first studied by \cite{Ch75} and popularized by \cite{ArGoGo90}. We refer to \cite{BaCh05} for an introduction to Stein's method. 
Stein's method has been widely used to prove limit theorems with error bounds in graph theory.
For example, \cite{ArGoGo90} and \cite{ChDiMe05} used Stein's method to prove Poisson limit theorems for monochromatic cliques in a uniformly colored complete graph.
\cite{CeFo06} considered more general monochromatic subgraphs counts when the distribution of colors was exchangeable. 
\cite{JaNo91} studied the asymptotic distribution of the number of copies of a given graph in various random graph models.
All of these results are obtained by exploiting the local dependence structure within random variables.
In addition to the local dependence structure, we also exploit the uncorrelatedness within $W_n$. This technique of exploiting the uncorrelatedness within random variables was also used in \cite{FaRo14} to obtain rates of convergence for the central limit theorem for subgraph counts in random graphs.

In the next section, we give the proof of Theorem \ref{t2}.

\section{Normal approximation}

Let $G=(V(G), E(G))$ be a simple undirected graph, where $V(G)$ is the vertex set and $E(G)$ is the edge set. Let $m=|E(G)|$ be the number of edges of $G$. We color each vertex of $G$ independently and uniformly at random with $c\geq 2$ colors. Formally, label the vertices of $G$ by $\{v_1,\dots, v_{|V(G)|}\}$ and denote the color of the vertex $v_i$ by $\xi_{v_i}$.
Label the edges of $G$ by $\{1,\dots, m\}$. 
For each edge $i$, we denote by $v_{i1}, v_{i2}$ the two vertices it connects, i.e., $i=(v_{i1}, v_{i2})$.
Without loss of generality, assume $deg(v_{i1})\leq deg(v_{i2})$ where $deg(v)$ denotes the degree of vertex $i$.
Using the above notation, the standardized number of monochromatic edges can be expressed as
\ben{\label{14}
W=\sum_{i=1}^n X_i:=\sum_{i=1}^m \frac{(\I(\xi_{v_{i1}}=\xi_{v_{i2}} )-\frac{1}{c})}{\sqrt{\frac{m}{c}(1-\frac{1}{c})}}.
}
Observing that $X_i$ and $X_j$ are uncorrelated if $i\ne j$, we have $\E W=0, \Var (W)=1$.

We will need the following lemmas in the proof of Theorem \ref{t2}.

\begin{lemma}\label{l1}
We have the following bounds on the moments of $X_i$:
\ben{\label{l1-1}
\E|X_i|\leq \frac{2}{\sqrt{mc}};\quad \E X_i^2=\frac{1}{m};\quad \E|X_i|^3 \leq \frac{\sqrt{c}}{m^{3/2}}.
}
\end{lemma}
\begin{proof}
The proof is elementary and therefore omitted.
\end{proof}

\begin{lemma}[Page 37 of \cite{BaHoJa92}]\label{l2}
For each edge $i=(v_{i1}, v_{i2})$, define $d_i=deg(v_{i1})\wedge deg(v_{i2})$. We have
\ben{\label{l2-1}
\sum_{i=1}^m d_i \leq \sqrt{2} m^{3/2}.
}
\end{lemma}

\begin{lemma}[Lemma 3.2 of \cite{BhDiMu13}]\label{l3}
The number of triangles, denoted by $\#(\Delta)$, in $G$ is bounded by $\sqrt{2} m^{2/3}/3$.
\end{lemma}

The following proposition is the key ingredient in proving Theorem \ref{t2}.

\begin{proposition}\label{p1}
For any function $f$ with bounded first and second derivatives, we have with $W$ defined in \eq{14},
\ben{\label{p1-1}
|\E f'(W)-\E W f(W)|\leq ||f''||(\frac{3}{2}\sqrt{\frac{c}{m}}+\frac{5\sqrt{2}}{\sqrt{c}})
+||f'|| \frac{2\cdot 2^{1/4}}{m^{1/4}}
}
where $||g||:=\sup_x |g(x)|$ for any function $g$.
\end{proposition}
\begin{proof}
For each edge $i=(v_{i1}, v_{i2})$ with $deg(v_{i1})\leq deg(v_{i2})$, define the neighborhood $N_i\subset \{1,\dots, m\}$ to be consisted of 
all the edges connects to $v_{i1}$.
Let
\be{
D_i=\sum_{j\in N_i} X_j,\quad   W_i=W-D_i.
}
Since the color of $v_{i1}$ is independent of $W_i$, we have $X_i$ is independent of $W_i$.
Therefore, by $\E X_i=0$, $\E X_i^2=1/m$, the Taylor expansion and adding and subtracting corresponding terms, we have
\besn{\label{12}
\E f'(W)-\E Wf(W)
&=\sum_{i=1}^m \E X_i^2 \E f'(W)  -\sum_{i=1}^m \E X_i[f(W)-f(W_i)]\\
&=\sum_{i=1}^m \E X_i^2 \E f'(W) -\sum_{i=1}^m \E X_i D_i f'(W-UD_i)\\
&=\sum_{i=1}^m \E X_i^2 \E f'(W) - \sum_{i=1}^m \E X_i^2 f'(W-UD_i) -\sum_{i=1}^m \E X_i (D_i-X_i) f'(W-UD_i)\\
&=:R_1-R_2-R_3-R_4
}
where $U$ is an independent random variable distributed uniformly in $[0,1]$ and
\be{
R_1=\sum_{i=1}^m \E X_i^2 \E f'(W) - \sum_{i=1}^m \E X_i^2 \E f'(W_i),
}
\be{
R_2=\sum_{i=1}^m \E X_i^2 [f'(W-UD_i)- f'(W_i)],
}
\be{
R_3=\sum_{i=1}^m \E X_i (D_i-X_i) [f'(W-UD_i)-f'(W)],
}
\be{
R_4= \E f'(W) \sum_{i=1}^m X_i(D_i-X_i).
}
First of all, by the Taylor expansion,
\be{
|R_1|\leq ||f''||\sum_{i=1}^m \E X_i^2 \E |D_i|\leq ||f''||\big( \sum_{i=1}^m \E |X_i|^3 
+\sum_{i=1}^m \sum_{j\in N_i\backslash \{i\}} \E |X_i|^2 \E |X_j|  \big).
}
By \eq{l1-1} and \eq{l2-1},
\be{
\sum_{i=1}^m \E |X_i|^3 \leq \sqrt{\frac{c}{m}},
\quad
\sum_{i=1}^m \sum_{j\in N_i\backslash \{i\}} \E |X_i|^2 \E |X_j| \leq \frac{1}{m}\frac{2}{\sqrt{mc}} \sum_{i=1}^m d_i\leq \frac{2\sqrt{2}}{\sqrt{c}}.
}
Therefore,
\be{
|R_1|\leq ||f''||(\sqrt{\frac{c}{m}}+\frac{2\sqrt{2}}{\sqrt{c}}).
}
By the same argument and the fact that $\{X_j: j\in N_i\}$ are jointly independent,
\be{
|R_2|\leq \frac{1}{2}||f''||(\sqrt{\frac{c}{m}}+\frac{2\sqrt{2}}{\sqrt{c}}).
}
For $R_3$, by the Taylor expansion,
\be{
|R_3|\leq \frac{1}{2}||f''||\sum_{i=1}^m \E |X_i| |D_i-X_i| |D_i|
\leq \frac{1}{2}||f''||\sum_{i=1}^m \E X_i^2 |D_i-X_i|+\frac{1}{2}||f''||\sum_{i=1}^m \E|X_i| |D_i-X_i|^2.
}
Again by \eq{l1-1} and \eq{l2-1} and the fact that $\{X_j: j\in N_i\}$ are jointly independent,
\be{
\sum_{i=1}^m \E X_i^2 |D_i-X_i| \leq \frac{1}{m} \frac{2}{\sqrt{mc}}\sum_{i=1}^m d_i \leq \frac{2\sqrt{2}}{\sqrt{c}},
}
\be{
\sum_{i=1}^m \E|X_i| |D_i-X_i|^2\leq \frac{2}{\sqrt{mc}} \frac{1}{m} \sum_{i=1}^m d_i\leq \frac{2\sqrt{2}}{\sqrt{c}}.
}
Therefore,
\be{
|R_3|\leq ||f''||\frac{2\sqrt{2}}{\sqrt{c}}.
}
Finallly we bound $|R_4|$. By the Cauchy-schwartz inequality, the fact that $\{X_j: j\in N_i\}$ are jointly independent and $\E X_i=0$,
\be{
|R_4|\leq ||f'||\sqrt{\Var(\sum_{i=1}^m X_i (D_i-X_i))}=|f'||\sqrt{\Var(\sum_{i=1}^m \sum_{j\in N_i\backslash \{i\}}X_i X_j)}.
}
Observe that if $j\in N_i\backslash \{i\}$ and $l\in N_k\backslash \{k\}$, $\Cov(X_i X_j, X_k X_l)=0$ unless $\{i,j\}=\{k,l\}$ or $\{i,j,k,l\}$ forms a triangle.
For the case $\{i,j\}=\{k,l\}$,
\be{
\Cov(X_i X_j, X_i X_j)=\frac{1}{m^2}.
}
For the case $\{i,j,k,l\}$ forms a triangle, with distinct $i,j,k$,
\be{
\Cov(X_i X_j, X_j X_k)=\E X_i X_j^2 X_k \leq \frac{1}{m^2}
}
where the last inequality is by straightforward calculation. Therefore, by \eq{l2-1}, Lemma \ref{l3}, and observing that each triangle in $G$ gives rise to $3$ ordered pairs of $(i,j)$ such that $j\in N_i\backslash \{i\}$, we have,
\be{
|R_4|\leq ||f'||\sqrt{2\sum_{i=1}^m \sum_{j\in N_i\backslash \{i\}} \frac{1}{m^2} +\frac{3\times 2}{m^2} \#(\Delta)}\leq 
||f'|| \frac{2\cdot 2^{1/4}}{m^{1/4}}.
}
The bound \eq{p1-1} follows from \eq{12} and the bounds on $|R_1|-|R_4|$.
\end{proof}

\begin{proof}[Proof of Theorem \ref{t2}]
By the definition of Wasserstein distance, we have
\be{
d_{W}(\mathcal{L}(W), N(0,1))=\sup_{||h'||\leq 1} |\E h(W)-\E h(Z)|.
}
where $Z$ is a standard Gaussian random variable.
Let $f_h$ be the solution to
\be{
f'(w)-wf(w)=h(w)-\E h(Z).
}
Replacing $w$ by $W$ and taking expectation on both sides of the above equation, we have
\ben{\label{15}
d_{W}(\mathcal{L}(W), N(0,1))=\sup_{||h'||\leq 1} |\E f_h'(W)-\E W f(W)|.
}
If $||h'||\leq 1$, then
it is know that (c.f. (2.14) of \cite{Ra04}, \cite{Do12})
\be{
||f_h'||\leq \sqrt{\frac{2}{\pi}},\quad ||f_h''||\leq 1.
}
The bound \eq{t2-1} is proved by \eq{15} and applying the above bounds in \eq{p1-1}.
\end{proof}

\begin{remark}\label{r1}
The following bound can be obtained following the proof of Theorem \ref{t2}:
\be{
d_{W}(\mathcal{L}(W), N(0,1))\leq C_0 (\sqrt{\frac{c}{m}}+\frac{K_m}{\sqrt{c}m^{3/2}}+\frac{1}{m^{1/4}})
}
where $C_0$ is an absolute constant, $K_m=\sum_{i=1}^m d_i$ and $d_i$ is defined in Lemma \ref{l2}.
For fixed $c$, the above error bound goes to zero if $m\to \infty$ and $K_m\ll m^{3/2}$.
This rules out complete graphs where $K_m\sim m^{3/2}$. Proposition 6.1 of \cite{BhDiMu13} gives this counter-example.
\end{remark}

\section*{Acknowledgments}

We thank the authors of the paper \cite{BhDiMu13} for suggesting using Stein's method to give an alternative proof of their result. This work was supported by the NUS-Overseas Postdoctoral Fellowship from the National University of Singapore.

\setlength{\bibsep}{0.5ex}
\def\bibfont{\small}


\end{document}